\documentclass[a4paper,12pt]{article}

\usepackage{amsmath, amsthm, amscd, amsfonts, amssymb, graphicx, color}
\usepackage[bookmarksnumbered, colorlinks, plainpages]{hyperref}

\usepackage{enumitem}
\usepackage{url}
\usepackage{multirow}
\usepackage{subfigure}
\usepackage[pass]{geometry}
\usepackage{cite}
\usepackage{cancel}

\usepackage{tikz}
\usetikzlibrary{calc,shapes,decorations.pathreplacing,matrix}

\tikzset{square matrix/.style={
    matrix of nodes,
    column sep=-\pgflinewidth, row sep=-\pgflinewidth,
    nodes={draw,
      minimum height=4.5pt,
      anchor=center,
      text width=4.5pt,
      align=center,
      inner sep=0pt
    },
  },
  square matrix/.default=1.2cm
}

\newtheorem{thm}{Theorem}

\newtheorem{prob}{Problem}
\newtheorem{ques}{Question}

\newtheorem{cor}{Corollary}

\newtheorem{prop}{Proposition}
\newtheorem{defn}{Definition}

\begin{document}

\title{Roman $k$-tuple domination in graphs}

\author{Adel P. Kazemi\\[1em]
Department of Mathematics\\ University of Mohaghegh Ardabili \\ P.O.\ Box 5619911367, Ardabil, Iran. \\
$^1$ Email: adelpkazemi@yahoo.com \\[1em]
}

\maketitle

\begin{abstract}
For any integer $k\geq 1$ and any graph $G=(V,E)$ with minimum degree at least $k-1$, we define a
function $f:V\rightarrow \{0,1,2\}$ as a Roman $k$-tuple dominating
function on $G$ if for any vertex $v$ with $f(v)=0$ there exist at least
$k$ and for any vertex $v$ with $f(v)\neq 0$ at least $k-1$ vertices in its neighborhood with $f(w)=2$. The minimum weight of a Roman $k$-tuple dominating function $f$ on $G$ is called the Roman $k$-tuple domination number of the graph where the weight of $f$ is $f(V)=\sum_{v\in V}f(v)$. 

In this paper, we initiate to study the Roman $k$-tuple
domination number of a graph, by giving some sharp bounds for the Roman $k$-tuple domination number of a garph, the Mycieleskian of a graph, and the corona graphs. Also finding the Roman $k$-tuple domination number of some known graphs is our other goal. Some of our results extend these one
given by Cockayne and et al. \cite{CDHH04} in 2004 for the Roman
domination number.
\\[0.2em]

\\[0.2em]

\end{abstract}

\pagestyle{myheadings}
\markboth{\centerline {\scriptsize  A. P. Kazemi}}   
{\centerline {\scriptsize A. P. Kazemi, Roman $k$-tuple domination in graphs}}

\section{ Introduction}

All graphs considered here are finite, undirected and simple. For
standard graph theory terminology not given here we refer to
\cite{West}. Let $G=(V,E) $ be a graph with the \emph{vertex set}
$V$ of \emph{order} $n(G)$ and the \emph{edge set} $E$ of
\emph{size} $m(G)$. The \emph{open neighborhood} of a vertex $v\in
V$ is $N_{G}(v)=\{u\in V\ |\ uv\in E\}$, while its cardinality is
the \emph{degree} of $v$. The \emph{closed neighborhood} of $v$ is
defined by $N_{G}[v]=N_{G}(v)\cup \{v\}$. Similarly, the \emph{open}
and \emph{closed neighborhoods} of a subset $X\subseteq V(G)$ are
$N_{G}(X)=\cup _{v\in X}N_{G}(v)$ and $N_{G}[X]=N_{G}(X)\cup X$,
respectively. The \emph{minimum} and \emph{maximum degree} of $G$
are denoted by $\delta =\delta (G)$ and $\Delta =\Delta (G)$,
respectively. If $\delta =\Delta=k$, then $G$ is called
$k$-\emph{regular}. We write $K_n$, $C_{n}$, $P_{n}$, and $W_n$ for
a \emph{complete graph}, a \emph{cycle}, a \emph{path}, and a
\emph{wheel} of order $n$, respectively, while $K_{n_1,...,n_p}$
denotes a \emph{complete $p$-partite graph}. Also $G[S]$ and
$\overline {G}$ denote the subgraph induced by a subset $S\subseteq
V$ and the \emph{complement} of $G$, respectively.

\vskip 0.15 true cm

For each integer $k\geq 1$, the $k$-\emph{join} $G\circ _{k}H$ of a graph $G$ to a graph $H$ of order at least $k$ is the graph obtained from the disjoint union of $G$ and $H$ by joining each vertex of $G$ to at least~$k$ vertices of $H$ \cite{HK}.

\vskip 0.15 true cm

Domination in graphs is now well studied in graph theory and the
literature on this subject has been surveyed and detailed in the two
books by Haynes, Hedetniemi, and Slater~\cite{HHS1, HHS2}. One type of domination is $k$-tuple domination number  that was introduced by Harary and Haynes \cite{HH}.

\begin{defn}
\emph{\cite{HH} For any positive integer $k$, a subset $S\subseteq V$ is a} $k$-tuple
dominating set \emph{of the graph $G$, if $| N_G[v]\cap S| \geq k$ for
every $v\in V$. The} $k$-tuple domination number $\gamma_
{\times k}(G)$ \emph{of $G$ is the minimum cardinality among the $k$-tuple
dominating sets of $G$.} 
\end{defn}

Henning and Kazemi in \cite{HK} introduced another type of domination called $k$-tuple total domination number of a graph which is an extension of the total domination number.

\begin{defn} 
\emph{\cite{HK} For any integer $k\geq 1$, a subset $S$ of $V$ is called a} $k$-tuple
total dominating set, \emph{abbreviated $k$TDS, of $G$ if for every vertex
$v\in V$, $|N(v)\cap S| \geq k$. The} $k$-tuple
total domination number $ \gamma _{\times k,t}(G)$ \emph{of $G$ is the
minimum cardinality of a $k$TDS of $G$. }
\end{defn}

Note that the 1-tuple domination number (1-tuple total domination number) is the classical domination number $\gamma(G)$ (total domination number $\gamma_t(G)$). A $k$-tuple dominating set ($k$-tuple total dominating set) of minimum cardinality of a graph $G$ is called a \emph{min}-$k$DS or  $\gamma_ {\times k}(G)$-\emph{set} (\emph{min}-$k$TDS or $\gamma_ {\times k,t}(G)$-\emph{set}).

\vskip 0.15 true cm

According to \cite{CDHH04}, Constantine the Great (Emperor of Rome)
issued a decree in the 4th century A.D. for the defense of his
cities. He decreed that any city without a legion stationed to
secure it must neighbor another city having two stationed legions.
If the first were attacked, then the second could deploy a legion to
protect it without becoming vulnerable itself. The objective, of
course, is to minimize the total number of legions needed. According
to it, Ian Steward by an article in Scientific American, entitled
"Defend the Roman Empire!" \cite{St99} suggested the Roman
dominating function.

\vskip 0.15 true cm

In \cite{KV}, K\"{a}mmerling and Volkmann extended the Roman
dominating function to the \emph{Roman $k$-dominating function} in
this way that for any vertex $v$ with $f(v)=0$ there are at least
$k$ vertices $w$ in its neighborhood with $f(w)=2$, and they defined the \emph{Roman $k$-domination number}
$\gamma_ {kR}(G)$ of a graph $G$ as the minimum weight of a Roman
$k$-dominating function $f$ on $G$ where the \emph{weight} of $f$ is $f(V)=\sum_{v\in V}f(v)$.

\vskip 0.15 true cm

This problem that for securing a city without a legion stationed or
a city with at least one legion stationed we need at least,
respectively, $k$ or $k-1$ cities having two stationed legions,
is our motivation to define the concept of Roman $k$-tuple domination number which is another extension of the Roman domination number.

\vskip 0.15 true cm

\begin{defn}
\emph{For any integer $k\geq 1$, a} Roman $k$-tuple dominating function, \emph{abbreviated R$k$DF, on
a graph $G$ with minimum degree at least $k-1$ is a function
$f\colon V\rightarrow \{0,1,2\}$ such that for any vertex $v$ with
$f(v)=0$ there exist at least $k$ and for any vertex $v$ with $f(v)\neq 0$ there exist
at least $k-1$ vertices $w$ in its neighborhood with $f(w)=2$.  The} Roman $k$-tuple domination
number $\gamma_ {\times kR}(G)$ \emph{of a graph $G$ is the minimum
weight of a Roman $k$-tuple dominating function $f$ on $G$ where the}
weight \emph{of $f$ is $f(V)=\sum_{v\in V}f(v)$.}
\end{defn}

The Roman 1-tuple domination number is the usual
\emph{Roman domination number} $\gamma_ {R}(G)$. 

\vskip 0.15 true cm

A \emph{min}-R$k$DF is a Roman $k$-tuple dominating function
with the minimum weight. For a Roman $k$-tuple dominating function
$f$ let $(V_0, V_1, V_2)$ be the ordered partition of $V$ induced
by $f$ where $V_i =\{v\in V \mid f(v)=i\}$ for $i = 0, 1, 2$. Since
there is a one-to-one correspondence between the function $f$ and
the ordered partitions $(V_0,V_1,V_2)$ of $V$, we will write
$f=(V_0,V_1,V_2)$. Figure \ref{fi:C_10)} shows a min-R2DF of cycle $C_{10}$.

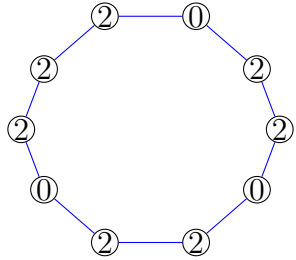
\begin{figure}[htp]
\centering
\begin{tikzpicture}
\tikzstyle{vertex}=[draw,circle,minimum size=10pt,inner sep=0pt]

\node[vertex] at (0,3) (r1) {$2$};
\node[vertex] at (0,0) (r2) {$2$};
\node[vertex] at (-0.8,2.3) (r3) {$2$};
\node[vertex] at (-0.8,0.7) (r4) {$0$};
\node[vertex] at (-1.1,1.5) (r5) {$2$};

\node[vertex] at (1.2,3) (r6) {$0$};
\node[vertex] at (1.2,0) (r7) {$2$};
\node[vertex] at (2,2.3) (r8) {$2$};
\node[vertex] at (2,0.7) (r9) {$0$};
\node[vertex] at (2.3,1.5) (r10) {$2$};

{
\color{blue}
\draw (r1) -- (r6); \draw (r6) -- (r8); \draw (r8) -- (r10); \draw (r10) -- (r9);
\draw (r9) -- (r7); \draw (r7) -- (r2); \draw (r2) -- (r4);
\draw (r4) -- (r5); \draw (r5) -- (r3); \draw (r3) -- (r1);
}
\end{tikzpicture}
\caption {$\gamma_{\times 2R}(C_{10})=14$} \label{fi:C_10)}
\end{figure}

\vskip 0.15 true cm

In this paper, we initiate to study the Roman $k$-tuple
domination number of a graph, by giving some sharp bounds for the Roman $k$-tuple domination number of a garph, the Mycieleskian of a graph, and the corona graphs. Also finding the Roman $k$-tuple domination number of some known graphs is our other goal. Some of our results extend these one
given by Cockayne and et al. \cite{CDHH04} in 2004 for the Roman
domination number.

\section {General results}

In this section, we state some properties of the Roman $k$-tuple dominating functions, and some sharp bounds for the Roman $k$-tuple domination number of a graph. 

\begin{prop}
\label{properties} For any min-R$k$DF $f=(V_0,V_1,V_2)$ on a graph $G$ with $\delta(G) \geq k-1$, the following statements hold.

\emph{(a)} $\gamma _{\times kR}(G)\geq \gamma _{kR}(G)$.

\emph{(b)} $V_1\cup V_2$ is a $k$-tuple dominating set of $G$.

\emph{(c)} $V_2$ is a $k$-tuple dominating set of $G[V_0\cup V_2]$.

\emph{(d)} For $k\geq 2$, $V_2$ is a $(k-1)$-tuple total dominating
set of $G$.

\emph{(e)} Every vertex of degree $k-1$ belongs to $V_1\cup V_2$.

\emph{(f)} $G[V_1]$ has maximum degree $1$.

\emph{(g)} Every vertex in $V_1$ is adjacent to precisely $k-1$ vertices in $V_2$.

\emph{(h)} Each vertex in $V_0$ is adjacent to at most two vertices in $V_1$.
\end{prop}

\begin{proof}
We omit the proofs of (a)-(e); they are clear. Let
$f=(V_0,V_1,V_2)$ be any $\gamma_{\times kR}$-function of $G$.

\vskip 0.2 true cm

(f) For any $x\in V_1$, let $x_1$, $x_2$, $\cdots$, $x_d$ be all
neighbors of $x$ in $V_1$. Since
\[
f'=(V_0\cup\{x_1,x_2,\cdots,x_d\},V_1-\{x,x_1,\cdots,x_d\},V_2\cup\{x\}),
\]
with the value $f'(V)=f(V)-d+1$, is a R$k$DF on $G$ if and only if
$d\leq 1$, we conclude that $G[V_1]$ has maximum degree 1.

\vskip 0.2 true cm

(g) For any $x\in V_1$, let $x_1$, $x_2$, $\cdots$, $x_d$ be all
neighbors of $x$ in $V_2$. Then $d\geq k-1$. If $d\geq k$ for some $x\in V_1$, then
$f'=(V_0\cup\{x\},V_1-\{x\},V_2)$ is a R$k$DF on $G$ with the value
$f'(V)=f(V)-1$, a contradiction. Therefore $d=k-1$.

\vskip 0.2 true cm

(h) For some $x\in V_0$, let $x_1$, $x_2$, $\cdots$, $x_d$ be all
neighbors of $x$ in $V_1$, for some $d\geq 3$. Then
$f'=(V_0\cup\{x_1,x_2,\cdots,x_d\},V_1-\{x_1,x_2,\cdots,x_d\},V_2\cup\{x\})$
is a R$k$DF on $G$ with the value $f'(V)=f(V)-d+2<f(V)$, a
contradiction. Therefore $d\leq 2$.
\end{proof}

As a consequence of Proposition \ref{properties} (c),(d), we have the following result.

\begin{cor}
If $G$ is a Roman $k$-tuple graph, that is $\gamma _{\times kR}(G)= 2\gamma
_{\times k}(G)$, then
\[
2\max\{\gamma_{\times (k-1),t}(G),\gamma_{\times k}(G)\} \leq
\gamma_{\times kR}(G).
\]
\end{cor}
For any graph $G=(V,E)$ of order $n$ and with minimum degree at least $k-1\geq 1$, since $(\emptyset,\emptyset,V)$ is a R$k$DF on $G$,
we have $\gamma _{\times kR}(G)\leq 2n$. On the other hand, since
for any R$k$DF $f=(V_0,V_1,V_2)$, $|V_2|\geq k$, we have $\gamma
_{\times kR}(G)\geq 2k$. Also, it can easily be verified that $\gamma
_{\times kR}(G)=2k$ if and only if $G=K_k$ or $G=H\circ_k K_k$ for
some graph $H$. Therefore we have proved the next theorem.

\begin{thm}
\label{2k=<gamaxkR=<2n} For any graph $G$ of order $n$ and with minimum degree at least $k-1\geq 1$,
\[
2k\leq \gamma _{\times kR}(G)\leq 2n,
\]
and $\gamma _{\times kR}(G)=2k$ if and only if $G=K_k$ or
$G=H\circ_k K_k$ for some graph $H$.
\end{thm}


Theorem \ref{2k=<gamaxkR=<2n} characterizes graphs $G$ with $\gamma
_{\times kR}(G)=2k$. Next proposition characterizes graphs $G$ with
$\gamma _{\times kR}(G)=2k+1$. First we construct a graph.

\vskip 0.15 true cm

Let $n\geq k+1\geq 3$. For $n=k+1$ let $\mathcal{A}_k$ be the complete graph $K_{k+1}$ minus an edge, and for $n>k+1$ let $\mathcal{A}_k$ be the graph with the vertex set
$V=\{v_i\mbox{ }|\mbox{ }1\leq i\leq n\}$ such that the induced
subgraph of $\mathcal{A}_k$ by $\{v_i \mbox{ }|\mbox{ } 1\leq i\leq
k+1\}$ is the complete graph $K_{k+1}$ minus edge $v_kv_{k+1}$, and
for any $i\geq k+2$, $\{v_j~|~1\leq j\leq k\}\subseteq
N_{\mathcal{A}_k}(v_i)$.

\begin{prop}
\label{gama xkR=2k+1} For any graph $G$ with $\delta(G)\geq k-1\geq 1$, $\gamma _{\times kR}(G)=2k+1$ if and only if $G\cong \mathcal{A}_k$, that is, $G$ is isomorphic to $\mathcal{A}_k$.
\end{prop}

\begin{proof}
Let $G$ be a graph with $\delta(G)\geq k-1\geq 1$. If $G\cong \mathcal{A}_k$, then $(V_0,V_1,V_2)$ is a min-R$k$DF on $G$ where $V_2=\{v_i\mbox{ }|\mbox{ }1\leq i\leq k\}$, $V_1=\{v_{k+1}\}$ and $V_0=V(\mathcal{A}_k)-V_1\cup V_2$, and so $\gamma _{\times kR}(G)=2k+1$.

\vskip 0.15 true cm

Conversely, let $\gamma _{\times kR}(G)=2k+1$, and let
$f=(V_0,V_1,V_2)$ be a min-R$k$DF on
$G$. Hence $|V_2|=k$ and $|V_1|=1$. If $V_2=\{v_i \mbox{ }|\mbox{ }
1\leq i\leq k\}$ and $V_1=\{v_{k+1}\}$, then the assumption $\gamma _{\times kR}(G)=2k+1$ implies that there exists a vertex in $V_2$, say $v_k$, which is not adjacent to $v_{k+1}$, that is $G\cong \mathcal{A}_k$.
\end{proof}

Note that if $k\geq 2$ and $G$ is ($k-1$)-regular, then $\gamma
_{\times kR}(G)=2n$. We will show that its converse holds only for
$k=2$. For $k\geq 3$, for example, if $G$ is a graph which is
obtained by the complete bipartite graph $K_{k,k}$ minus a
matching of cardinality $k-1$, then $\gamma _{\times
kR}(G)=4k$ while $G$ is not ($k-1$)-regular.

\begin{prop}
\label{gamax2R=2n} For any graph $G$ of order $n$ and without isolate
vertex, $ \gamma _{\times 2R}(G)=2n$ if and only if $G=\ell
K_2$ for some $\ell\geq 1$.
\end{prop}

\begin{proof}
Let $G=(V,E)$ be a graph of order $n$ and without isolate vertex, and let
$\gamma _{\times 2R}(G)=2n$. Since $deg(w)\geq 2$, for some vertex $w$, implies that the function
$(\{w\},\emptyset,V-\{w\})$ is a R2DF on $G$ with weight less than $2n$, we conclude $G=\ell K_2$ for some $\ell\geq 1$. Since the proof of inverse case is trivial, we have completed our proof.
\end{proof}

Cockayne and et al. in \cite{CDHH04} proved that for any graph $G$,
\begin{equation}%
\gamma(G)\leq \gamma _{R}(G)\leq 2\gamma(G).
\label{gama=<gama{R}=<2gamma}
\end{equation}
As an extension of inequality (\ref{gama=<gama{R}=<2gamma}), next theorem improves the lower bound $2k$ given in Theorem \ref{2k=<gamaxkR=<2n} for $k\geq 2$.

\begin{thm}
\label{xk+k, xkR, 2xk} For any graph $G$ with $\delta (G)\geq
k-1\geq 1$,
\[
\gamma _{\times k}(G)+k\leq \gamma _{\times kR}(G)\leq 2\gamma
_{\times k}(G),
\]
and the lower bound is sharp.
\end{thm}

\begin{proof}
Since for any min-$k$DS $S$ of $G=(V,E)$, the function $f =(V
-S,\emptyset,S)$ is a R$k$DF on $G$, we have $\gamma _{\times
kR}(G)\leq 2\mid S\mid = 2\gamma _{\times k}(G)$. On the other hand, since for any min-R$k$DF $f=(V_0,V_1,V_2)$ on $G$, $V_1\cup V_2$ is a $k$-tuple dominating set of $G$, we have
\[
\gamma _{\times kR}(G)=2|V_2|+|V_1| \geq \gamma _{\times k}(G)+|V_2|\geq
\gamma _{\times k}(G)+k.
\]

\vskip 0.15 true cm

For any graph $H$ of order $k$, the lower bound is sharp for $G=\overline{K_k}\circ_{*(k-1)}K_k$. Because the function $(\emptyset, V(\overline{K_k}),V(K_k))$ is a min-R$k$DF on $G$ and $V(K_k)$ is a min-$k$DS of $G$.
\end{proof}


Following E. J. Cockayne, P. A. Dreyer Jr., S. M. Hedetniemi and S. T.
Hedetniemi \cite{CDHH04}, we will say that a graph $G$ is a
\emph{Roman $k$-tuple graph} if $\gamma _{\times kR}(G)= 2\gamma
_{\times k}(G)$. Next proposition characterizes the Roman $k$-tuple graphs.
\begin{prop}
\label{k-TupRomGraph} A graph $G$ with $\delta (G)\geq k-1$ is a
Roman $k$-tuple graph if and only if it has a min-R$k$DF $f=(V_0,\emptyset,V_2)$, that is, $V_2$ is a min-$k$DS of $G$.
\end{prop}

\begin{proof}
Let $G$ be a Roman $k$-tuple graph, and let $S$ be a min-$k$DS of $G$. Since $f=(V-S,\emptyset,S)$ is a R$k$DF on $G$ with weight $f(V)=2\mid S\mid= 2\gamma _{\times k}(G)= \gamma _{\times kR}(G)$, we conclude that $f$ is a min-R$k$DF.

\vskip 0.15 true cm

Conversely, if $f=(V_0,\emptyset,V_2)$ is a min-R$k$DF on $G$, then $\gamma _{\times kR}(G)=2\mid V_2\mid$, and $V_2$ is a $k$DS of $G$. Hence $\gamma _{\times k}(G)\leq |V_2| = \gamma _{\times kR}(G)/2$. Applying Theorem \ref{xk+k, xkR, 2xk} implies $\gamma _{\times kR}(G)=2\gamma _{\times k}(G)$, that is, $G$ is a Roman $k$-tuple graph.
\end{proof}

\begin{cor}
\label{1-TupRomGraph} \emph{\cite{CDHH04}} A graph $G$ is a Roman graph if
and only if it has a min-RDF $f=(V_0,\emptyset,V_2)$.
\end{cor}


\section {Complete bipartite graphs, paths, cycles and wheels}

Here, we calculate the Roman $k$-tuple domination number of
a complete bipartite graph, a cycle, a path, and a wheel.


\begin{prop}
\label{R$k$DF,Comp.Bipartite} For any integer $n\geq m\geq k-1\geq 1$,
\begin{equation*}
\gamma_{\times kR}(K_{n,m})=\left\{
\begin{array}{ll}
3k-3+n & \mbox{if }n\geq m=k-1, \\
4k-2 & \mbox{if }n\geq m=k, \\
4k-1 & \mbox{if } n=m=k+1, \\
4k & \mbox{if } n> m\geq k+1.
\end{array}
\right.
\end{equation*}
\end{prop}

\begin{proof}
Assume that $V(K_{n,m})$ is partitioned to the independent sets $X$
and $Y$ such that $|X|=n$ and $|Y|=m$. Since the Roman $k$-tuple
dominating functions given in each of the following cases have
minimum weight, our proof is completed.

\vskip 0.15 true cm

\textsc{Case 1.} $n\geq m=k-1$. Consider $f=(\emptyset,\emptyset,X\cup Y)$ when $n=m$, and consider $f=(\emptyset,V_1,V_2)$ when $n>m$ in which
$Y\subseteq V_2$, $|V_2\cap X|=k-1$ and $V_1=X-V_2$.
\vskip 0.15 true cm

\textsc{Case 2.} $n\geq m=k$. Consider $f=(V_0,\emptyset,V_2)$ where $|V_2\cap Y|=k$, $|V_2\cap X|=k-1$ and
$V_0=X\cup Y-V_2$.

\vskip 0.15 true cm

\textsc{Case 3.} $n=m=k+1$. Consider $f=(V_0,V_1,V_2)$ where $|V_2\cap Y|=k$, $|V_2\cap X|=k-1$, $V_1=Y-V_2$ and
$V_0=X\cup Y-V_1\cup V_2$.

\vskip 0.15 true cm

\textsc{Case 4.} $n>m\geq k$. Consider
$f=(V_0,\emptyset,V_2)$ where $|V_2\cap X|=|V_2\cap Y|=k$ and
$V_0=X\cup Y-V_2$.
\end{proof}

\begin{cor}
\label{gamm xkR(Kn,m)=k gamm R(Kn,m) if...} If $n> m\geq k+1\geq 3$, then $\gamma_{\times kR}(K_{n,m})=k\gamma_R(K_{n,m})$.
\end{cor}

\begin{proof}
It is sufficient to consider 
\begin{equation*}
\gamma_{\times R}(K_{n,m})=\left\{
\begin{array}{ll}
2 & \mbox{if }n\geq m=1, \\
3 & \mbox{if }n\geq m=2, \\
4 & \mbox{if } n\geq m\geq 3.
\end{array}
\right.
\end{equation*}
\end{proof}

In the next step, we will calculate $\gamma_{\times 2R}(C_n)$ (notice $\gamma_{\times 3R}(C_n)=2n$ by Proposition \ref{properties}).

\begin{prop}
\label{gamm x2R(Cn)} For any cycle $C_n$ of order $n\geq 3$,
$\gamma_{\times 2R}(C_n)=2\lceil\frac{2n}{3}\rceil$.
\end{prop}

\begin{proof}
Let $V(C_n)=\{1,2,...,n\}$, and let $E(C_n)=\{ij\mbox{ }|\mbox{ }
j\equiv i+1 \pmod{n},~1\leq i\leq n\}$. Since $(V_0,\emptyset,V(C_n)-V_0)$ is a R2DF on $C_n$ where $V_0=\{3t+1\mbox{
}|\mbox{ }0\leq t\leq \lfloor \frac{n}{3}\rfloor-1\}$, we have $\gamma_{\times2R}(C_n)\leq 2\lceil\frac{2n}{3}\rceil$.

\vskip 0.15 true cm

On the other hand, since in any R2DF every three consecutive vertices have
at least weight four, we have $\gamma_{\times 2R}(C_n)\geq
\lceil\frac{4n}{3}\rceil$. Since $\lceil\frac{4n}{3}\rceil=2\lceil\frac{2n}{3}\rceil$ where $n\not\equiv 2 \pmod{3}$, we consider
$n\equiv 2 \pmod{3}$. Then $\lceil\frac{4n}{3}\rceil=2\lceil\frac{2n}{3}\rceil-1$. Now let
$f=(V_0,V_1,V_2)$ be a min-R2DF on $C_n$. Since
every vertex in $V_2$ is adjacent to at least one vertex in $V_2$
and $f$ has minimum weight, we conclude that if
$i-1,i\in V_2$, then $i+1\in V_0$, as possible as. Therefore $f(3t+1)=0$ and
$f(3t)=f(3t+2)=2$ for $0\leq t\leq \lfloor\frac{n}{3}\rfloor-1$.
This implies $f(n-2)=f(n-1)=2$, and so $\gamma_{\times
2R}(C_n)=f(V(C_n))=\lceil\frac{4n}{3}\rceil+1=2\lceil\frac{2n}{3}\rceil$.
\end{proof}

\begin{prop}
\label{gamm x2R(Pn)} For any path $P_n$ of order $n\geq 2$,
\begin{equation*}
\gamma_{\times 2R}(P_n)=\left\{
\begin{array}{ll}
2\lceil\frac{2n}{3}\rceil    & \mbox{if }n\equiv 1,2 \pmod{3}, \\
2\lceil\frac{2n}{3}\rceil +1 & \mbox{if }n=3, \\
2\lceil\frac{2n}{3}\rceil +2 & \mbox{otherwise}.
\end{array}
\right.
\end{equation*}
\end{prop}

\begin{proof}
Let $V(P_n)=\{1,2,...,n\}$, and let $E(P_n)=\{ij\mbox{ }|\mbox{ }
j=i+1,~1\leq i\leq n-1 \}$. Since $(\emptyset,\emptyset,V(P_n))$ is the only min-R2DF on $P_2$ and $(\emptyset,\{1\},\{2,3\})$ is a min-R2DF on $P_3$, we consider $n\geq 4$. Let $f=(V_0,V_1,V_2)$ be a min-R2DF on $P_n$. Then $f(1)=f(n)=1$, and
$f(2)=f(3)=f(n-2)=f(n-1)=2$. This implies $\gamma_{\times
2R}(P_4)=6$, $\gamma_{\times 2R}(P_5)=8$, $\gamma_{\times
2R}(P_6)=10$, as desired. Therefore, we may assume $n\geq 7$. Let $\mathcal{L}=V(P_n)-\{1,2,3,n-2,n-1,n\}$.
Since every three consecutive vertices in $\mathcal{L}$ have at least weight four and every two consecutive vertices in
it have at least weight two, we conclude that
$(V_0,V_1,V_2)$ is a min-R2DF on $P_n$ where $V_0=\{3t+1~|~1\leq t \leq
\lfloor\frac{n-1}{3}\rfloor-1\}$, $V_1=\{1,n\}$, $V_2=V(P_n)-V_0\cup V_1$ and $n\geq 7$, and this completes our proof.
\end{proof}

Since it can be easily verified that for any $n\geq 3$,
\begin{equation*}
\gamma_{\times 2}(C_n)=\left\{
\begin{array}{ll}
\lceil\frac{2n}{3}\rceil       & \mbox{if }n \mbox{ is odd}, \\
\lfloor\frac{2n}{3}\rfloor   & \mbox{if }n \mbox{ is even},
\end{array}
\right.
\end{equation*}
and for any $n\geq 2$,
\begin{equation*}
\gamma_{\times 2}(P_n)=\left\{
\begin{array}{ll}
\lceil\frac{2n}{3}\rceil       & \mbox{if }n \equiv 0,2,5,8 \pmod{9}, \\
\lceil\frac{2n}{3}\rceil +1   & \mbox{otherwise}.
\end{array}
\right.
\end{equation*}
we have the following two propositions by Propositions \ref{gamm
x2R(Cn)} and \ref{gamm x2R(Pn)}.

 \begin{prop}
\label{Cn is 2RG} For any $n\geq 3$, a cycle $C_n$ is a Roman 2-tuple graph if and only if $n \not\equiv 2,4 \pmod{6}$.
\end{prop}
 \begin{prop}
\label{Pn is 2RG} For any $n\geq 2$, a path $P_n$ is a Roman 2-tuple graph if and only if $n\neq 3$ and $n \not\equiv 0,1,4,7 \pmod{9}$.
\end{prop}

We recall that $W_n$ denotes a wheel of order $n\geq 4$ with $V(W_n)=\{v_0,v_1,\cdots,v_{n-1}\}$ such that
$deg(v_0)=n-1$ and $deg(v_i)=3$ for $1\leq i \leq n$. Here, we
calculate $\gamma_{\times kR}(W_n)$ for $1\leq k \leq 4$, because
$\delta(W_n)=3\geq k-1$. Since $\gamma_{R}(W_n)=2$ and
$\gamma_{\times 4R}(W_n)=2n$, we
consider $k=2,3$ in the next two propositions. First we state a proposition from \cite{KV}.

\begin{prop}
\label{gamm kR(Wn)} \emph{\cite{KV}} For any wheel $W_n$ of order $n\geq4$,
\begin{equation*}
\gamma_{kR}(W_n)=\left\{
\begin{array}{ll}
2 & \mbox{if }k=1, \\
\lceil\frac{2(n-1)}{3}\rceil +2& \mbox{if }k=2, \\
n & \mbox{if }k\geq 3.
\end{array}
\right.
\end{equation*}
\end{prop}

\begin{prop}
\label{gamm x2R(Wn)} For any wheel $W_n$ of order $n\geq4$,
$\gamma_{\times 2R}(W_n)=\lceil\frac{2(n-1)}{3}\rceil +2$.
\end{prop}

\begin{proof}
Let $X=\{v_{3t+1}~|~0\leq t\leq \lfloor\frac{n-1}{3}\rfloor-1\}\cup\{v_0\}$ and let $V_0=V(W_n)-(V_1\cup V_2)$ in a R2DF $(V_0,V_1,V_2)$ on $W_n$. Since 
\begin{equation*}
f=(V_0,V_1,V_2)=\left\{
\begin{array}{ll}
(V_0,\emptyset,X\cup\{v_{n-2}\})  & \mbox{if }n\equiv 0 \pmod{3}, \\
(V_0,\emptyset,X)                   & \mbox{if }n\equiv 1 \pmod{3}, \\
(V_0,\{v_{n-2}\},X)                       & \mbox{if }n\equiv 2 \pmod{3},
\end{array}
\right.
\end{equation*}
is a R2DF on $W_n$ with weight $\lceil\frac{2(n-1)}{3}\rceil +2$, we obtain
$\gamma_{\times 2R}(W_n)=\lceil\frac{2(n-1)}{3}\rceil +2$, by
Propositions \ref{properties}-(a) and \ref{gamm kR(Wn)}.
\end{proof}

\begin{prop}
\label{gamm x3R(Wn)}  For any wheel $W_n$ of order $n\geq4$, $\gamma_{\times 3R}(W_n)=2n-2\lfloor \frac{n-1}{3}\rfloor$.
\end{prop}

\begin{proof}
Let $f=(V_0,V_1,V_2)$ be a minimal R3DF on $W_n$. Since every vertex, except probably $v_0$, has degree three and $v_0$ is adjacent to all other $n-1$ vertices, we conclude $v_0\in V_2$. Also, we know $v_i\in V_2$, for some $1\leq i \leq n-1$, implies $v_{i-1},v_{i+1}\in V_2$. By considering these facts and the minimality of the weight of $f$, we obtain $V_1=\emptyset$, and $|V_0|\leq \lfloor \frac{n-1}{3} \rfloor$. Hence $\gamma_{\times 3R}(W_n)\geq 2|V_2|=2n-2|V_0|\geq 2n-2\lfloor \frac{n-1}{3}\rfloor$. On the other hand, since $(\{v_{3t+1}~|~0\leq t\leq \lfloor\frac{n-1}{3}\rfloor-1\}, \emptyset,V(W_n)-V_0)$ is a R3DF on $W_n$ with weight $2n-2\lfloor \frac{n-1}{3}\rfloor$, we obtain $\gamma_{\times 3R}(W_n)=2n-2\lfloor \frac{n-1}{3}\rfloor$.
\end{proof}

\section{Mycieleskian of a graph}

In this section, we give some shap bounds for the Roman $k$-tuple domination number of the Mycieleskian of a graph in terms of the same number of the graph and $k$. Also we present the Roman $k$-tuple domination number of the Mycieleskian of the complete graphs. First we recall the definition of Mycieleskian of a graph.

\begin{defn}
\emph{\cite{West} Th}e Mycieleskian $M(G)$ \emph{of a graph $G=(V,E)$ is a graph
with vertex set $V\cup U\cup \{w\}$, and edge set $E\cup \{u_jv_i~|~ v_jv_i\in E \mbox{ and } u_j\in U \}\cup \{u_jw~|~
u_j\in U\}$ where $U=\{u_{j}~|~ v_{j}\in V\}$}.
\end{defn}

Figure \ref{fi:Myc(K_5)} shows the Mycileskian of $K_5$. 
\begin{figure}[htp]
\centering
\begin{tikzpicture}
\tikzstyle{vertex}=[draw,circle,minimum size=10pt,inner sep=0pt]

\node[vertex] at (0,3) (r1) {$v_1$};
\node[vertex] at (0,0) (r2) {$v_2$};
\node[vertex] at (-0.8,2.3) (r3) {$v_3$};
\node[vertex] at (-0.8,0.7) (r4) {$v_4$};
\node[vertex] at (-1.2,1.5) (r5) {$v_5$};

\node[vertex] at (1.2,3) (c1) {$u_1$};
\node[vertex] at (1.2,0) (c2) {$u_2$};
\node[vertex] at (2,2.3) (c3) {$u_3$};
\node[vertex] at (2,0.7) (c4) {$u_4$};
\node[vertex] at (2.4,1.5) (c5) {$u_5$};

\node[vertex] at (4,1.5) (c6) {$w$};

{
\color{blue}
\draw (r1) -- (r2); \draw (r1) -- (r3); \draw (r1) -- (r4); \draw (r1) -- (r5);
\draw (r2) -- (r3); \draw (r2) -- (r4); \draw (r2) -- (r5);
\draw (r3) -- (r4); \draw (r3) -- (r5);
\draw (r4) -- (r5);
}
{
\color{red}
\draw (c1) -- (r2); \draw (c1) -- (r3); \draw (c1) -- (r4); \draw (c1) -- (r5);
\draw (c2) -- (r1); \draw (c2) -- (r3); \draw (c2) -- (r4); \draw (c2) -- (r5); 
\draw (c3) -- (r1); \draw (c3) -- (r2); \draw (c3) -- (r4); \draw (c3) -- (r5);
\draw (c4) -- (r1); \draw (c4) -- (r2); \draw (c4) -- (r3); \draw (c4) -- (r5);
\draw (c5) -- (r1); \draw (c5) -- (r2); \draw (c5) -- (r3); \draw (c5) -- (r4); 
\draw (c6) -- (c1); \draw (c6) -- (c2); \draw (c6) -- (c3); \draw (c6) -- (c4); \draw (c6) -- (c5);
}
\end{tikzpicture}
\caption{The Mycileskian of $K_5$}\label{fi:Myc(K_5)}
\end{figure}
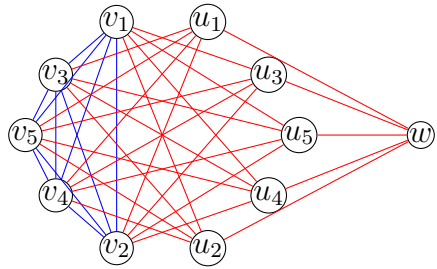


\begin{thm}
\label{RkDF(G) =<RkDF(M(G))=<RkDF(G)+2k} For any graph $G$ with $\delta(G)\geq k-1\geq 1$,
\[
\gamma_{\times kR}(G)+\min\{k-1,2\} \leq \gamma_{\times kR}(M(G))\leq \gamma_{\times kR}(G)+2k.
\]
\end{thm}

\begin{proof}
Let $G$ be a graph with $\delta(G)\geq k-1\geq 1$ and the vertex set $V=\{v_i~|~1\leq i\leq n\}$. 
Since for any min-R$k$DF $f=(V_0,V_1,V_2)$ on $G$, the function $g=(W_0,W_1,W_2)$ is a R$k$DF on $M(G)$ with weight
$\gamma_{\times kR}(G)+2k$ where $W_2=V_2\cup U'\cup \{w\}$ (for some subset $U'$ of $U$ of cardinality $k-1$), $W_1=V_1$ and
$W_0=V_0\cup (U-U')$, we obtain $\gamma_{\times kR}(M(G))\leq
\gamma_{\times kR}(G)+2k$.

\vskip 0.15 true cm

Now let $f=(V_0,V_1,V_2)$ be a min-R$k$DF on $M(G)$
such that $|V_1\cap U|$ and $|V_2\cap U|$ is as possible as minimum.
 Let $L=\{i~|~u_i\in V_1\}$, $L'=\{i~|~v_i\in V_1\}$, $T=\{i~|~u_i\in V_2\}$, and $T'=\{i~|~v_i\in V_2\}$ where $|T|=t\geq k-1$(because of $N_{M(G)}(w)=U$), $|T'|=t'$, $|L|=\ell$, and $|L'|=\ell'$. In the following three cases we show $\gamma_{\times kR}(M(G)) \geq \gamma_{\times kR}(G)+\min\{k-1,2\}$, and our proof will be completed.

\vskip 0.15 true cm

\textsc{Case 1.} $w\in V_0$. Then $t\geq k$ and 
\begin{equation*}
|N_{M(G)}(v_i)\cap V_2\cap V|\left\{
\begin{array}{ll}
=k-1  & \mbox{if }i\in L, \\
\geq k-1     & \mbox{if }i\in T, \\
\geq k                     & \mbox{if }i\not\in L\cup T.
\end{array}
\right.
\end{equation*}
Let 
\[
L_0=\{v_i\in V_0~|~i\in L\}\cup \{v_i\in V_0~|~i\in T, \mbox{ and }|N_{M(G)}(v_i)\cap V_2\cap V|=k-1\}
\]
be a set of cardinality $\ell_0$. Then $\ell\leq \ell_0 \leq \ell+t$. By choosing $V_2'=V_2\cap V$,
$V_1'=(V_1\cap V)\cup L_0$, $V_0'=V-(V_1'\cup V_2')$, since $f'=(V_0',V_1',V_2')$ is a R$k$DF on $G$, we have
\[
\begin{array}{lcl}
\gamma_{\times kR}(G) & \leq & f'(V) \\
                                    & =    & \gamma_{\times kR}(M(G))+\ell_0 -\ell -2t.                                    
\end{array}
\]
Hence
\[
\begin{array}{lcl}
\gamma_{\times kR}(M(G)) & \geq & \gamma_{\times kR}(G)+2t+\ell-\ell_0\\
                                        & \geq& \gamma_{\times kR}(G)+t\\
                                        & \geq& \gamma_{\times kR}(G)+k.
\end{array}
\]

\vskip 0.15 true cm

\textsc{Case 2.} $w\in V_1$. Then $t=k-1$, and $\ell\leq 1$. Because
if $\ell\geq 2$, then by choosing $V_1'=V_1\cap V$, $V_2'=V_2\cup
\{w\}$, $V_0'=V(M(G))-V_1'\cup V_2'$ the function
$f'=(V_0',V_1',V_2')$ is a R$k$DF on $M(G)$, and so
\[
\begin{array}{lcl}
\gamma_{\times kR}(M(G)) & \leq & f'(V) \\
                         & =    & 2(|V_2|+1)+(|V_1|-1)-|U\cup V_1|\\
                         & =    & \gamma_{\times kR}(M(G))+1-\ell,
\end{array}
\]
implying that $\ell\leq 1$. Hence 
\begin{equation*}
|N_{M(G)}(v_i)\cap V_2\cap V|\left\{
\begin{array}{ll}
\geq k-1     & \mbox{if }i\in T\cup L, \\
\geq k                     & \mbox{if }i\not\in L\cup T.
\end{array}
\right.
\end{equation*}
Let
\[
L_1=\{v_i\in V_0~|~i\in T\cup L, \mbox{ and } |N_{M(G)}(v_i)\cap V_2\cap V|=k-1\}
\]
be a set of cardinality $\ell_1$. Hence $\ell_1\leq k$. Then $f'=(V_0',V_1',V_2')$ is a R$k$DF on $G$ where
$V_2'=V_2\cap V$, $V_1'=(V_1\cap V)\cup L_1$, $V_0'=V-(V_1'\cup
V_2')$, and so
\[
\begin{array}{lcl}
\gamma_{\times kR}(G) & \leq & f'(V) \\
                      & =    & 2|V_2|+|V_1|-2k+1+\ell_1-\ell\\
                      & =    & \gamma_{\times kR}(M(G))-2k+1+\ell_1-\ell.
\end{array}
\]
Hence
\[
\begin{array}{lcl}
\gamma_{\times kR}(M(G)) & \geq & \gamma_{\times kR}(G)+2k-\ell_1+\ell-1\\
                         & \geq & \gamma_{\times kR}(G)+k-1.
\end{array}
\]
\vskip 0.15 true cm

\textsc{Case 3.} $w\in V_2$. (Notice that we may assume that there is no
min-R$k$DF $g$ on $M(G)$ with $g(w)\neq 2$.) Then
\begin{equation*}
|N_{M(G)}(v_i)\cap V_2\cap V|\left\{
\begin{array}{ll}
\geq k-2            & \mbox{if }i\in T\cup L, \\
\geq k-1            & \mbox{if }i\not\in L\cup T.
\end{array}
\right.
\end{equation*}

\vskip 0.15 true cm

\textsc{Subcase 3.1} $T\cap T'=\emptyset$. Then the function $f'=(V_0',V_1',V_2')$ is a R$k$DF on $G$ where $V_2'=(V_2\cap V)\cup \{v_i~|~i\in T\}$, $V_1'=(V_1\cap V)-\{ v_i~|~ i \in T,~v_i\in V_1\}$ and $V_0'=V-(V_1'\cup V_2')$, and so
\[
\begin{array}{lcl}
\gamma_{\times kR}(G) & \leq & f'(V) \\
                      & =    & 2|V_2|+|V_1|-f(U)-f(w)+2t-|T\cap L'| \\
                      & =    & \gamma_{\times kR}(M(G))-\ell-2-|T\cap L'|\\
                      & \leq & \gamma_{\times kR}(M(G))-2,
\end{array}
\]
which implies $\gamma_{\times kR}(M(G)) \geq \gamma_{\times kR}(G)+2$.

\vskip 0.15 true cm

\textsc{Subcase 3.2} $T\cap T'\neq \emptyset$. Let $f''$ be a function which is obtained from $f'$ in Subcase 3.1 by adding some needed vertices from $N_G[v_i]$ to $V_2'$ or $V_1'$ if 
\begin{equation*}
|N_G(v_i)\cap V_2|<\left\{
\begin{array}{ll}
k   & \mbox{if }f'(v_i)=0, \\
k-1   & \mbox{if }f'(v_i)\neq 0
\end{array}
\right.
\end{equation*}
(this is possible because $|N_G[v_i]|\geq k$). Let $f''(V(G))-f'(V(G))=p$. Then $f''$ is a R$k$DF on $G$, and so
\[
\begin{array}{lcl}
\gamma_{\times kR}(G) & \leq & f''(V) \\
                      & =  & \gamma_{\times kR}(M(G))-f(U)-f(w)+2|T-T'|-|L'\cap T|+p\\
                      & =  & \gamma_{\times kR}(M(G))-\ell-2-2|T\cap T'|-|T\cap L'|+p\\
                      & \leq & \gamma_{\times kR}(M(G))-2.
\end{array}
\]
The last inequality is obtained from the facts that $p\leq 2t$, $|T\cap T'|+|T\cap L'|\leq |T|=t$, and
 $|T \cap T'|\leq t$. Hence $\gamma_{\times kR}(M(G)) \geq \gamma_{\times kR}(G)+2$.
\end{proof}

By the fact $\gamma_{\times kR}(K_n)=2k$, the next theorem states that the upper bound given in Theorem
\ref{RkDF(G) =<RkDF(M(G))=<RkDF(G)+2k} is sharp.

\begin{thm}
\label{R$k$DF,M(Kn)} For any $n\geq k\geq 2$, $\gamma_{\times
kR}(M(K_n))=4k$.
\end{thm}

\begin{proof}
Let $V(K_n)=\{v_i~|~1\leq i \leq n\}$, and let $V(M(K_n))=V\cup
U\cup \{w\}$. Let $f=(V_0,V_1,V_2)$ be a min-R$k$DF on
$M(K_n)$. We show that $f(V(M(K_n)))\geq 4k$.
Since $N_{M(K_n)}(w)=U$ and $N_{M(K_n)}(u_i)\subseteq V\cup\{w\}$
for each $u_i\in U$, we have $|V_2\cap U|\geq k-1$ and $|V_2\cap
V|\geq k-1$. Let $V_2\cap V=\{v_i~|~i\in I\}$ and $V_2\cap
U=\{u_i~|~i\in J\}$ for some $I,J\subseteq \{1,2,...,n\}$. Then
\[
f(V(M(K_n)))=2(|I|+|J|)+f(w)+f(U-V_2)+f(V-V_2).
\]

\vskip 0.15 true cm

\textsc{Case 1.} $|J|=k-1$. Then $w\in V_1\cup
V_2$. First let $|I|=k-1$. Then $U-V_2\subseteq V_1$, and so
\[
\begin{array}{lcl}
f(V(M(K_n))) & \geq & 4(k-1)+1+2(n-k+1) \\
             & =    & 4k-3+2(n-k+1).
\end{array}
\]
Since $n\leq 2k-3$ implies $u_i,v_i\in V_2$ for
some $i\in J$, and so $|N_{M(K_n)}(u_i)\cap V_2|< k-1$, we have $n\geq 2k-2$. Hence for $k\geq 3$,
\[
\begin{array}{lcl}
f(V(M(K_n))) & \geq & 4k-3+2(n-k+1) \\
             & \geq & 4k-3+2(k-1) \\
             & =    & 6k-5 \\
             & \geq & 4k.
\end{array}
\]
Let $k=2$. If $n=2$, then $M(K_2)=C_5$, and so $\gamma_{\times
2R}(M(K_2))=2\lceil \frac{10}{3}\rceil=8=4k$, by Proposition
\ref{gamm x2R(Cn)}. If $n\geq 3$, then $n-k+1\geq 2$, and so
\[
\begin{array}{lcl}
f(V(M(K_n))) & \geq & 4k-3+2(n-k+1) \\
             & \geq & 4k+1.
\end{array}
\]
Now let $|I|\geq k$. Since $f$ has minimum weight, we have $|I|=k$, and so
\[
f(V(M(K_n)))=2k+2(k-1)+f(w)+f(V_1\cap U).
\]
If $V_1\cap U=\emptyset$, then $U\cap V_0=U-V_2$.
Since every vertex in $U\cap V_0$ must be adjacent to all vertices in
$V_2\cap V$, we have $V_2\cap V\subseteq \{v_i~|~i\in J\}$, which is
not possible. Therefore $V_1\cap U\neq \emptyset$, and so $f(V(M(K_n)))
\geq 4k$.

\vskip 0.15 true cm

\textsc{Case 2.} $|J|\geq k$. Then
\[
f(V(M(K_n)))\geq 2(|I|+|J|)+f(U-V_2)+f(V-V_2)+f(w).
\]
Since $|J|\geq k+1$ or $|I|\geq k$ impily $f(V(M(K_n))) \geq 4k$, we assume
$|J|=k$ and $|I|=k-1$. This implies $I\cap J=\emptyset$, and so
$n\geq 2k-1$. On the other hand, $|I|=k-1$ implies
$U-V_2\subseteq V_1$, and so $f(U-V_2)\geq |U|-k=n-k\geq k-1$.
Therefore, if $k\geq 3$, then
\[
\begin{array}{lcl}
f(V(M(K_n))) & \geq & 2(2k-1)+k-1 \\
             & =    & 5k-3 \\
             & \geq & 4k,
\end{array}
\]
and if $k=2$, then $\{v_i~|~i\in J\}\subseteq V_1$ which implies $f(V(M(K_n)))\geq 5k-3+2\geq 4k$.

\vskip 0.15 true cm

Finally, by choosing a subset $W_2\subseteq V(M(K_n))$ with this
property that $|W_2\cap V|=|W_2\cap U|=k$, and $W_0=V(M(K_n))-W_2$, the
function $(W_0,\emptyset,W_2)$ is a R$k$DF on $M(K_n)$ with weight
$4k$, implying that $\gamma_{\times kR}(M(K_n))=4k$ (Figure \ref{fi:M(K_5)} shows some min-R3DFs for $K_5$ and $M(K_5)$).


\begin{figure}[htp]
\centering
\begin{tikzpicture}
\tikzstyle{vertex}=[draw,circle,minimum size=10pt,inner sep=0pt]

\node[vertex] at (0,3) (r1) {$2$};
\node[vertex] at (0,0) (r2) {$2$};
\node[vertex] at (-1.1,1.5) (r5) {$2$};
\node[vertex] at (1.8,2.3) (r8) {$0$};
\node[vertex] at (1.8,0.7) (r9) {$0$};

{
\color{blue}
\draw (r1) -- (r2); \draw (r1) -- (r5); \draw (r1) -- (r8); \draw (r1) -- (r9);
\draw (r2) -- (r5); \draw (r2) -- (r8); \draw (r2) -- (r9);
\draw (r5) -- (r8); \draw (r5) -- (r9); \draw (r8) -- (r9);
}
\end{tikzpicture}
\qquad\qquad
\begin{tikzpicture}
\tikzstyle{vertex}=[draw,circle,minimum size=10pt,inner sep=0pt]

\node[vertex] at (0,3) (r1) {$2$};
\node[vertex] at (0,0) (r2) {$2$};
\node[vertex] at (-0.8,2.3) (r3) {$0$};
\node[vertex] at (-0.8,0.7) (r4) {$0$};
\node[vertex] at (-1.2,1.5) (r5) {$2$};

\node[vertex] at (1.2,3) (c1) {$2$};
\node[vertex] at (1.2,0) (c2) {$2$};
\node[vertex] at (2,2.3) (c3) {$0$};
\node[vertex] at (2,0.7) (c4) {$0$};
\node[vertex] at (2.4,1.5) (c5) {$2$};

\node[vertex] at (4,1.5) (c6) {$0$};

{
\color{blue}
\draw (r1) -- (r2); \draw (r1) -- (r3); \draw (r1) -- (r4); \draw (r1) -- (r5);
\draw (r2) -- (r3); \draw (r2) -- (r4); \draw (r2) -- (r5);
\draw (r3) -- (r4); \draw (r3) -- (r5);
\draw (r4) -- (r5);
}
{
\color{red}
\draw (c1) -- (r2); \draw (c1) -- (r3); \draw (c1) -- (r4); \draw (c1) -- (r5);
\draw (c2) -- (r1); \draw (c2) -- (r3); \draw (c2) -- (r4); \draw (c2) -- (r5); 
\draw (c3) -- (r1); \draw (c3) -- (r2); \draw (c3) -- (r4); \draw (c3) -- (r5);
\draw (c4) -- (r1); \draw (c4) -- (r2); \draw (c4) -- (r3); \draw (c4) -- (r5);
\draw (c5) -- (r1); \draw (c5) -- (r2); \draw (c5) -- (r3); \draw (c5) -- (r4); 
\draw (c6) -- (c1); \draw (c6) -- (c2); \draw (c6) -- (c3); \draw (c6) -- (c4); \draw (c6) -- (c5);
}
\end{tikzpicture}
\caption{$\gamma_{\times 3R}(K_{5})=6$ (left), and $\gamma_{\times 3R}(M(K_5))=12$ (right)}\label{fi:M(K_5)}
\end{figure}
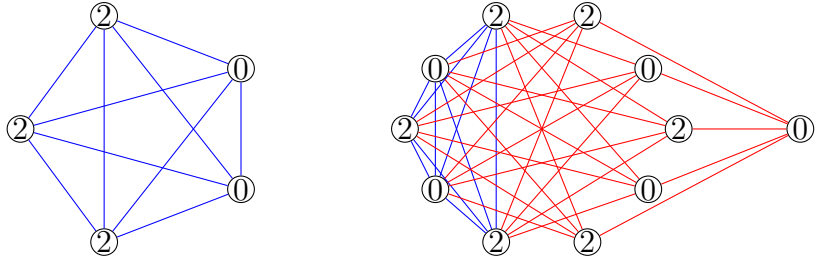
\end{proof}


\section {The corona graphs}

Here, we study the Roman $k$-tuple domination number of corona graphs. We recall that for any graphs $G$ and $H$ of orders $n$ and $m$, respectively, the \emph{corona graph} $cor(G,H)$ is a graph obtained from $G$ and
$H$ by taking one copy of $G$ and $n$ copies of $H$ and joining with
an edge each vertex from the $i$-th copy of $H$ with the $i$-th
vertex of $G$. Hereafter, in $cor(G,H)$ we will denote the set of vertices of $G$ by $V=\{v_1,v_2,\cdots,v_n\}$ and the $i$-th copy of $H$ by $H_i=(W_i,E_i)$. 

\vskip 0.15 true cm

First we give some bounds for the $k$-tuple domination number of a corona graph in the next theorem.

\begin{thm}
\label{LUB k-tupl DS of Corona} For any
graphs $G$ and $H$ with $\delta(H)\geq k-2\geq 0$,
\[
k|V(G)| \leq \gamma_{\times k}(cor(G,H))\leq (|V(H)|+1)|V(G)|,
\]
and these bounds are sharp, and $\gamma_{\times k}(cor(G,H))=k|V(G)|$ if and only if  $H=K_{k-1}$ or $H=F \circ_{k-1} K_{k-1}$ for some graph $F$.
\end{thm}

\begin{proof}
Since for any $k$DS $S$ of $cor(G,H)$ and any vertex $w$ in $H_i$, $|N_{cor(G,H)}[w]\cap
S|\geq k$, and on the other hand since $V(cor(G,H))$ is a $k$DS of $cor(G,H)$, we have
\[
k|V(G)| \leq \gamma_{\times k}(cor(G,H))\leq (|V(H)|+1)|V(G)|.
\]
Obviously, $\gamma_{\times k}(cor(G,H))=k|V(G)|$ if and only if
$H=K_{k-1}$ or $H=F \circ_{k-1} K_{k-1}$ for some graph $F$. For the
upper bound, if $H$ is a ($k-2$)-regular graph, then
$\gamma_{\times k}(cor(G,H))=(|V(H)|+1)|V(G)|$.
\end{proof}

\begin{thm}
\label{LUB kRDF of Corona} For any graphs
$G$ and $H$ with $\delta(H)\geq k-1\geq 1$,
\[
2k|V(G)|\leq \gamma_{\times kR}(cor(G,H))\leq 2\gamma_{\times
k}(cor(G,H)).
\]
\end{thm}

\begin{proof} Since Theorem \ref{xk+k, xkR, 2xk} implies $\gamma_{\times kR}(cor(G,H))\leq 2\gamma_{\times
k}(cor(G,H))$, it is sufficient to prove the lower bound. Let $f=(V_0,V_1,V_2)$ be a R$k$DF on
$cor(G,H)$ and let $v_i$ be a vertex of $G$. In all of the
following cases we prove $f(V)\geq 2k|V(G)|$, and so $\gamma_{\times kR}(cor(G,H))\geq 2k|V(G)|$ (we recall that for any subset $T\subseteq V$, $f(T)=\sum_{v\in T}f(v)$).

\vskip 0.15 true cm

\textrm{Case 1. } $f(v_i)=0$. If there exists a vertex $v\in W_i\cap
V_0$, then $|N_{W_i}(v)\cap V_2|\geq k$, and so $f(W_i\cup
\{v_i\})\geq 2k$. If there exists a vertex $v\in W_i\cap V_1$, then
$|N_{W_i}(v)\cap V_2|\geq k-1$. Now $k\geq 2$ implies that there exists
a vertex $v'\in N_{W_i}(v)\cap V_2$, and so $|N_{W_i}(v')\cap
V_2|\geq k-1$. Therefore, $|(N_{W_i}(v)\cup N_{W_i}(v'))\cap
V_2|\geq k$ which implies $f(W_i\cup \{v_i\})\geq 2k+1$. Finally,
if for any $v'\in W_i$ we have $f(v')=2$, then $f(W_i\cup
\{v_i\})\geq 2k$.

\vskip 0.15 true cm

\textrm{Case 2. } $f(v_i)=1$. If there exists a vertex $v\in W_i\cap
V_0$, then $|N_{W_i}(v)\cap V_2|\geq k$, and so $f(W_i\cup
\{v_i\})\geq 2k+1$. If there exists a vertex $v\in W_i\cap V_1$, then
$|N_{W_i}(v)\cap V_2|\geq k-1$. Now $k\geq 2$ implies that there exists
a vertex $v'\in N_{W_i}(v)\cap V_2$, and so $|N_{W_i}(v')\cap
V_2|\geq k-1$. Therefore $|(N_{W_i}(v)\cup N_{W_i}(v'))\cap
V_2|\geq k$ which implies $f(W_i\cup \{v_i\})\geq 2k+2$. Finally,
if for any $v'\in W_i$ we have $f(v')=2$, then $f(W_i\cup
\{v_i\})\geq 2k+1$.

\vskip 0.15 true cm

\textrm{Case 3. } $f(v_i)=2$. If there exists a vertex $v\in W_i\cap
V_0$, then $|N_{W_i}(v)\cap V_2|\geq k-1$, and so $f(W_i\cup
\{v_i\})\geq 2k$. If there exists a vertex $v\in W_i\cap V_1$, then
$|N_{W_i}(v)\cap V_2|\geq k-2$, and so $f(W_i\cup \{v_i\})\geq 2k-1$.
Since $f(W_i\cup \{v_i\})=2k-1$ if and only if $H=K_{k-1}$, we
obtain $f(W_i\cup \{v_i\})\geq 2k$. Finally, if for any $v'\in W_i$ we have $f(v')=2$, then $f(W_i\cup \{v_i\})\geq
2k$.
\end{proof}

The following theorem is obtained by Theorems \ref{LUB k-tupl DS of
Corona} and \ref{LUB kRDF of Corona}.

\begin{thm}
\label{exact LB kRDF of Corona} For any graphs $G$ and $H$ with $\delta(H)\geq k-2\geq 0$, $\gamma_{\times
kR}(cor(G,H))=2k|V(G)|$ if and only if $H=K_{k-1}$ or
$H=F\circ_{k-1}K_{k-1}$ for some graph $F$.
\end{thm}


\section{Some questions and problems}

Finally, we end our paper with some useful questions and problems.
\begin{ques}
Is $M(G)$ a Roman $k$-tuple graph if $G$ is a Roman $k$-tuple graph?
\end{ques} 
\begin{ques}
For any Roman $k$-tuple graph $G$, is there a Roman $k$-tuple graph $H$ such that $G=M(H)$?
\end{ques} 
\begin{ques}
Find graphs $G$ whose Roman $k$-tuple domination number achieves the bounds in Theorem \ref{RkDF(G) =<RkDF(M(G))=<RkDF(G)+2k}?
\end{ques}
 
\begin{ques}
For any graph $G$, whether $\gamma_{\times 2R}(G)\geq 2\gamma_R(G)$?
\end{ques}
\begin{prob}
Find $\gamma_{\times kR}(M(C_n))$ for $2\leq k\leq 4$ and $\gamma_{\times kR}(M(P_n))$ for $2\leq k\leq 3$.
\end{prob}

\begin{prob}
Find the Roman $k$-tuple domatic number of a graph.
\end{prob}

\begin{prob}
Characterize graphs $G$ with $\gamma_{\times
2R}(G)=\gamma_R(G)$.
\end{prob}
\begin{prob}
Characterize graphs $G$ with $\gamma_{\times kR}(G)=\gamma_{kR}(G)$.
\end{prob}

In \cite{LC}, the authors have defined the \emph{total Roman dominating function} on a graph $G$ as a Roman domination function $f=(V_0,V_1,V_2)$ on it with this additional property that the induced subgraph $G[V_1\cup V_2]$ has no isolated vertex, and in a similar way, they have defined the \emph{total Roman domination number} $\gamma_{tR}(G)$ of $G$. Since $\gamma_{tR}(G)\leq \gamma_{\times kR}(G) \leq \gamma_{\times (k+1)R}(G)$ for any $k\geq 2$, we have
\begin{equation}%
\gamma_{tR}(G)\leq \gamma_{\times 2R}(G).
\label{gama_{tR}(G) =< gama_{X2R}(G)}
\end{equation}%
So,
\begin{prob}
Finding graphs $G$ satisfying $\gamma_{\times 2R}(G)=\gamma_{tR}(G)$ is a natural problem.
\end{prob}


\end{document}